\documentclass[11pt]{amsart}

\usepackage{amsmath}
\usepackage{amssymb,mathtools,amsthm,latexsym}

\calclayout

\title{Pointwise inequalities for Sobolev functions on generalized cuspidal domains}
\author{Zheng Zhu}

\address{Zheng Zhu\\
Department of Mathematics and Statistics\\
University of Jy\"askyl\"a, P.O. Box 35 (MaD),
FI-40014, Jyv\"askyl\"a, Finland}
\email{\tt zheng.z.zhu@jyu.fi}

\usepackage{stackrel}

\usepackage{color}

\numberwithin{equation}{section}
\long\def\colred#1\endred{{\color{red}#1}}
\long\def\colgreen#1\endgreen{{\color{green}#1}}
\long\def\colmagenta#1\endmagenta{{\color{magenta}#1}}
\long\def\colblue#1\endblue{{\color{blue}#1}}
\long\def\colyellow#1\endyellow{{\color{yellow}#1}}

\usepackage{tikz}
\usetikzlibrary{matrix,arrows}

\theoremstyle{plain}
\newtheorem{theorem}[equation]{Theorem}

\newtheorem{lemma}[equation]{Lemma}
\newtheorem{corollary}[equation]{Corollary}

\theoremstyle{remark}

\theoremstyle{definition}

\newtheorem{defn}[equation]{Definition}

\newtheorem*{question*}{Question}

\usepackage{graphicx}


\thanks{The author has been supported by the Academy of Finland via Centre of Excellence in Analysis and Dynamics Research (Project \#323960). He would like to thank Prof P. Koskela and Prof. S. Eriksson-Bique for some useful discussion.}

\newcounter{prob}
\def\rr{{\mathbb R}}
\def\rn{{{\rr}^n}}

\def\fz{\infty}

\def\boz{{\Omega}}

\def\bint{{\ifinner\rlap{\bf\kern.25em--}
\int\else\rlap{\bf\kern.45em--}\int\fi}\ignorespaces}

\def\bbint{{\ifinner\rlap{\bf\kern.25em--}
\hspace{0.078cm}\int\else\rlap{\bf\kern.45em--}\int\fi}\ignorespaces}

\def\lp{{L^{p}(\boz)}}

\def\r{\right}
\def\lf{\left}

\def\setcolon{;}

\def\XXint#1#2#3{{\setbox0=\hbox{$#1{#2#3}{\int}$ }
\vcenter{\hbox{$#2#3$ }}\kern-.58\wd0}}

\newcommand{\co}{\mskip0.5mu\colon\thinspace}   % Colon for maps.

\def\vint_#1{\mathchoice%
        {\mathop{\kern 0.2em\vrule width 0.6em height 0.69678ex depth -0.58065ex
                \kern -0.8em \intop}\nolimits_{\kern -0.4em#1}}%
        {\mathop{\kern 0.1em\vrule width 0.5em height 0.69678ex depth -0.60387ex
                \kern -0.6em \intop}\nolimits_{#1}}%
        {\mathop{\kern 0.1em\vrule width 0.5em height 0.69678ex
            depth -0.60387ex
                \kern -0.6em \intop}\nolimits_{#1}}%
        {\mathop{\kern 0.1em\vrule width 0.5em height 0.69678ex depth -0.60387ex
                \kern -0.6em \intop}\nolimits_{#1}}}
\def\vintslides_#1{\mathchoice%
        {\mathop{\kern 0.1em\vrule width 0.5em height 0.697ex depth -0.581ex
                \kern -0.6em \intop}\nolimits_{\kern -0.4em#1}}%
        {\mathop{\kern 0.1em\vrule width 0.3em height 0.697ex depth -0.604ex
                \kern -0.4em \intop}\nolimits_{#1}}%
        {\mathop{\kern 0.1em\vrule width 0.3em height 0.697ex depth -0.604ex
                \kern -0.4em \intop}\nolimits_{#1}}%
        {\mathop{\kern 0.1em\vrule width 0.3em height 0.697ex depth -0.604ex
                \kern -0.4em \intop}\nolimits_{#1}}}

\begin{document}

\maketitle
\begin{center}
 \textit{In memory of Prof. Jan Mal\'y (Honza)}
\end{center}

\begin{abstract}
Let $\boz\subset\rr^{n-1}$ be a bounded star-shaped domain and $\boz_\psi$ be an outward cuspidal domain with base domain $\boz$. Then, we prove for $1<p\leq\fz$, that $W^{1, p}(\boz_\psi)=M^{1, p}(\boz_\psi)$ if and only if $W^{1, p}(\boz)=M^{1, p}(\boz)$\\
%\colmagenta
%Haj\l{}asz Sobolev or Haj\l{}asz-Sobolev?
%\endmagenta
\end{abstract}

\section{Introduction}
A Sobolev function $u$ on $\rn$ satisfies the pointwise inequality 
\[|u(z_1)-u(z_2)|\leq |z_1-z_2|(CM(|\nabla u|)(z_1)+CM(|\nabla u|)(z_2))\]
at every Lebesgue points of $u$, where $M(|\nabla u|)$ is the Hardy-Littlewood maximal function of $|\nabla u|$, see \cite{AandF, BandH, Pitor, Lewis}. Motived by this fact,  Haj\l{}asz defined the so-called Haj\l{}asz-Sobolev space $M^{1, p}(U)$ which consists of all $u\in L^p(U)$ with a nonnegative $g\in L^p(U)$ such that for every $z_1, z_2\in U\setminus E$ with $|E|=0$, we have
\[|u(z_1)-u(z_2)|\leq|z_1-z_2|(g(z_1)+g(z_2)),\]
where $U\subset\rn$ is a domain. For all domains $U$ and any $p\in[1, \fz]$, one always has $M^{1, p}(U)\subset W^{1, p}(U)$. Furthermore, when $p=1$, the inclusion is strict, see \cite{Pitor, ks}. Since $M^{1, p}(U)=W^{1, p}(U)$ implies that every $u\in W^{1, p}(U)$ supports global Poincar\'e inequality on a bounded domain $U$, see \cite{Pitor}, it is a natural question to ask
\begin{center}
\textit{For which domains $U\subset\rn$ do we have $M^{1,p}(U)=W^{1,p}(U)$?}
\end{center}

In this note, we concentrate on a class of generalized outward cuspidal domains. Let $\boz\subset\rr^{n-1}$ be a bounded star-shaped domain with star-center $x_o\in\boz$. Let $\psi\colon (0,1]\to (0,\fz)$ be a left continuous and non-decreasing function. We consider outward cuspidal domains of the form
\begin{eqnarray}\label{cusp}
\boz_\psi&:=&\lf\{(t, x)\in (0, 1)\times\rr^{n-1}\setcolon\frac{x-x_o}{\psi(t)}+x_o\in\boz \r\}\\
&\cup&\lf\{(t, x)\in[1,2)\times\rr^{n-1}\setcolon \frac{x-x_o}{\psi(1)}+x_o\in\boz\r\},\nonumber
\end{eqnarray}
\begin{figure}[h!tbp]
\centering
\includegraphics[width=0.5\textwidth]
{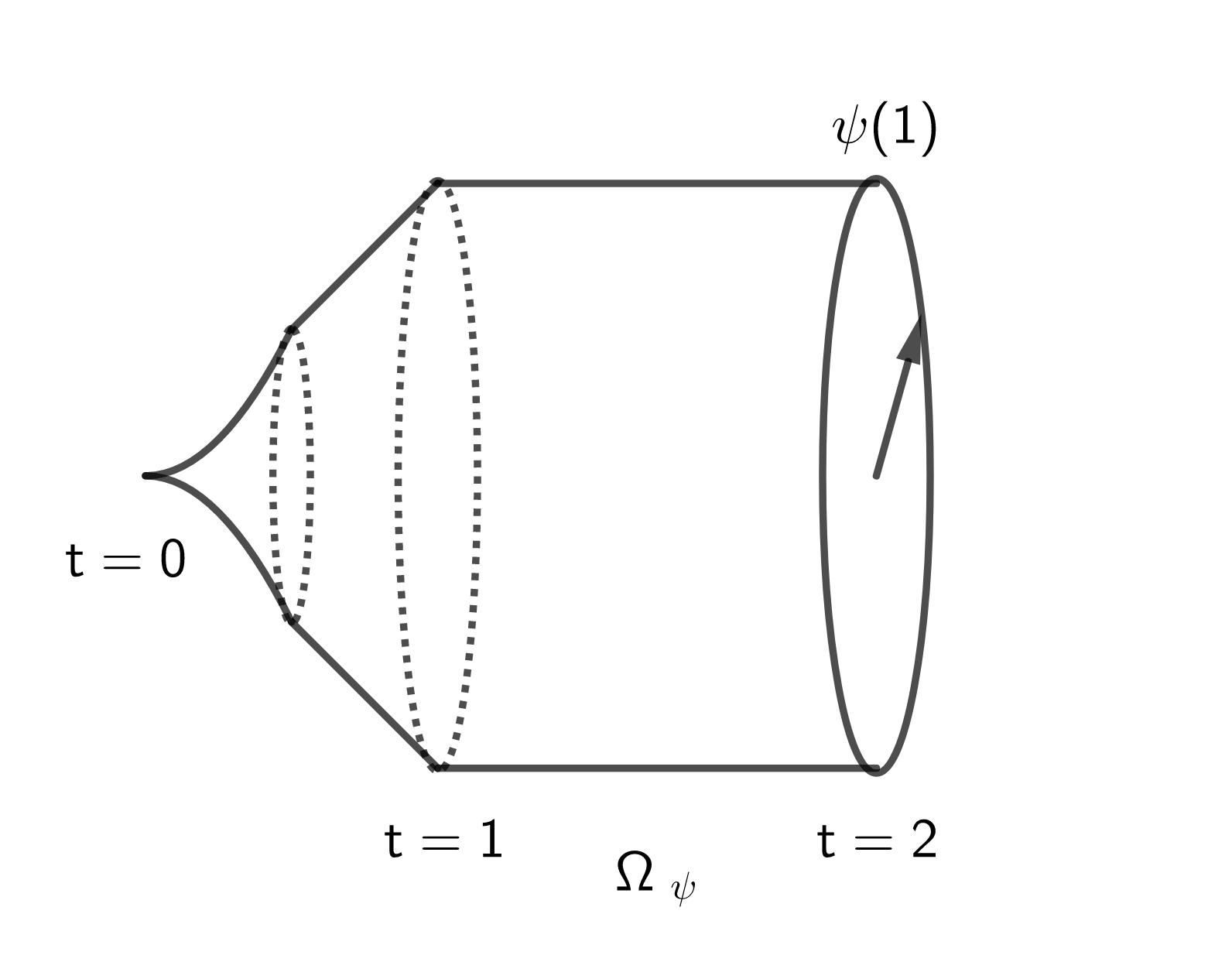}
\label{fig:cusps}
\end{figure}
We require $\psi$ is left-continuous is just to get that $\boz_\psi$ is open. The $(n-1)$-dimensional bounded star-shaped domain $\boz$ is called the base domain of the outward cuspidal domain $\boz_\psi$. The star-shapeness is necessary to guarantee that $\boz_\psi$ is always a domain for every left continuous and increasing function $\psi$. One can easily see that $\boz_\psi$ is also star-shaped. Since we only concentrate on the cuspidal singularity on the tip, the seemingly strange annexes are included for some technical convenience in the argument. 
 
In next theorem, we will show that the classical Sobolev space coincides with the Haj\l{}as-Sobolev space on an outward cuspidal domain $\boz_\psi$ if and only if these two function spaces coincide on the base domain $\boz$. This theorem implies that even on a very irregular domain $U\subset\rn$, we can have $M^{1, p}(U)=W^{1, p}(U)$, for example see Corollary \ref{cor2} below. Hence, it is difficult may even impossible to give a characterization to those domains where the classical Sobolev space and Haj\l{}asz-Sobolev spaces coincide.
\begin{theorem}\label{thm:main}
Let $\boz\subset\rr^{n-1}$ be a bounded star-shaped domain and $\psi:(0, 1]\to(0, \fz)$ be a left continuous and increasing function. Define the corresponding cuspidal domain $\boz_\psi$ as in \eqref{cusp}. Then $W^{1,p}(\boz_\psi)=M^{1,p}(\boz_\psi)$ if and only if $W^{1, p}(\boz)=M^{1, p}(\boz)$ for $1<p\leq\fz$.
\end{theorem}
By \cite{Pitor}$, M^{1, p}(U)$ can always be embedded into $W^{1, p}(U)$, for arbitrary $1\leq p\leq\fz$ and arbitrary domain $U\subset\rn$. Hence, if one can show they coincide as sets, the classical $Open\ Mapping\ Theorem$ will imply that the corresponding norms of a fix element are comparable up to a uniform constant. In \cite{Romanov}, Romanov showed $W^{1, p}(\boz_\psi)=M^{1, p}(\boz_\psi)$ if $\boz\subset\rr^{n-1}$ is a unit ball, $\psi(t)=t^s$ with $s>1$ and $p>\frac{1+(n-1)s}{n}$. The main result in \cite{IMRN} told us that the above restriction on $p$ is superfluous. To be more precise, the authors showed if $\boz\subset\rr^{n-1}$ is an unit ball, $M^{1, p}(\boz_\psi)=W^{1, p}(\boz_\psi)$ for arbitrary $1<p\leq\fz$ and arbitrary left continuous and increasing function $\psi$. Since unit ball is a bounded star-shaped domain for $W^{1, p}(B^{n-1}(0, 1))=M^{1, p}(B^{n-1}(0, 1))$ for every $1<p\leq\fz$, the result in \cite{IMRN} is a special case of Theorem \ref{thm:main} here.

A domain $U\subset\rn$ is called a $W^{1, p}$-extension domain for $1\leq p\leq\fz$, if for every $u\in W^{1, p}(U)$, there exists an extension function $E(u)\in W^{1, p}(\rn)$ with $E(u)\big|_\boz\equiv u$ and 
\[\|E(u)\|_{W^{1, p}(\rn)}\leq C\|u\|_{W^{1, p}(U)}\]
for a positive constant $C$ independent of $u$. By the classical result due to Haj\l{}asz \cite{Pitor}, if $U\subset\rn$ is a $W^{1, p}$-extension domain for $1<p\leq\fz$, then $W^{1,p}(U)=M^{1, p}(U)$. Hence, we have the following corollary to Theorem \ref{thm:main}.
\begin{corollary}\label{cor}
Let $\boz\subset\rr^{n-1}$ be a bounded star-shaped $W^{1, p}$-extension domain for $1<p\leq\fz$ and $\psi:(0, 1]\to(0, \fz)$ be a left continuous and increasing function. Then we have $W^{1, p}(\boz_\psi)=M^{1, p}(\boz_\psi)$, where $\boz_\psi$ is the corresponding outward cuspidal domain defined in (\ref{cusp}).
\end{corollary}
Our theorem enables a large class of new examples. In particular, the following corollary shows how the cuspidal construction can be iterated to give examples which have different cuspidal sigularities in every coordinate directions. We thank Sylvester Eriksson-Bique for pointing out this application.
\begin{corollary}\label{cor2}
Let $I\subset\rr$ be an interval which contains $0$ and $\{\psi_i:(0,1]\to(0, \fz)\}_{i=1}^n$ be a class of left continuous and increasing function. Set $\boz_1:=I$ and define a sequence of outward cuspidal domains $\{\boz_i\subset\rr^i\}_{i=2}^n$ inductively by setting
\begin{eqnarray}
\boz_i&:=&\lf\{(t, x)\in(0, 1)\times\rr^{i-1}; \frac{x-x_o^{i-1}}{\psi_i(t)}+x_o^{i-1}\in\boz_{i-1}\r\}\nonumber\\
        &\cup&\lf\{(t, x)\in[1, 2)\times\rr^{i-1}; \frac{x-x_o^{i-1}}{\psi_i(1)}+x_o^{i-1}\in\boz_{i-1}\r\},\nonumber
\end{eqnarray}
where $x_o^{i-1}\subset \boz_{i-1}$ is a star center. Then $W^{1, p}(\boz_n)=M^{1, p}(\boz_n)$ for every $1<p\leq\fz$.
\end{corollary}
\begin{proof}
By induction, for every $i\in\{1, 2, \cdots, n-1\}$, $\boz_i$ is a star-shaped domain with $W^{1, p}(\boz_i)=M^{1, p}(\boz_i)$. Hence, it is a direct corollary of Theorem \ref{thm:main}.
\end{proof}

\section{Definitions and Preliminaries}
In what follows, $U\subset\rn$ is always a domain and $\boz\subset\rn$ is always a bounded star-shaped domain. We denote $C^\fz(\overline{U})$ to be the restriction of $C^\fz(\rn)$ on $\overline U$ by setting 
\[C^\fz(\overline U):=\{u\big|_{\overline U}: u\in C^\fz(\rn)\}.\] 
For a measurable subset $E\subset\rn$, $\chi_E$ is the corresponding characteristic function and $|E|$ means the $n$-dimensional Hausdorff measure of $E$. Typically, $c$ or $C$ will be constants that depend on various parameters and may vary even on the same line of inequalities. The Euclidean distance between points $x, y$ in the Euclidean space $\rn$ is denoted by $|x-y|$. The open $m$-dimensional ball of radius $r$ centered at the point $x$ is denoted by $B^{m}(x, r)$. For two points $x, y\in\rn$, $[x, y]$ means the segment starting from $x$ to $y$.
\begin{defn}\label{de:star}
A domain $\boz\subset\rn$ is said to be star-shaped, if there exists a point $x_o\in\boz$ such that for every $x\in\boz$, the segment $[x, x_o]$ between $x$ and $x_o$ is contained in $\boz$. The point $x_o$ is called the star center of $\boz$.
\end{defn} 
For a star-shaped domain, the set of star centers may not be unique. For example, for a convex domain, every point inside the domain is a star center. From now on, whenever we mention a star-shaped domain $\boz$, it means we already fix a star center $x_o\in\boz$. Let $\boz\subset\rn$ be a bounded star-shaped domain. For every $0<\lambda<\fz$, we define 
$$\lambda\boz:=\lf\{x\in\rn: \frac{x-x_o}{\lambda}+x_o\in\boz\r\}.$$ 
We write 
\begin{equation}
\rn=\rr\times\rr^{n-1}:=\lf\{z:=(t, x)\in\rr\times\rr^{n-1}\r\}\, .\nonumber
\end{equation}
Let $\boz\subset\rr^{n-1}$ be a bounded star-shaped domain. We consider a left continuous and increasing function $\psi\colon (0,1]\to (0,\fz)$, extend the definition of $\psi$ to the interval $(0, 2)$ by setting 
\begin{equation}
\psi(t)=\psi(1), \ \ {\rm for\ every}\ \ t\in (1,2) \nonumber
\end{equation}
and write
\begin{equation}
\boz_\psi=\lf\{(t, x)\in (0, 2)\times\rr^{n-1}\setcolon \frac{x-x_o}{\psi(t)}+x_o\in\boz\r\}\, .\nonumber
\end{equation}

The space of locally integrable functions is denoted by $L^1_{\rm loc}(U)$. For every measurable set $Q\subset U$ with $0<|Q|<\infty$, and every non-negative measurable or integrable function $f$ on $Q$ we define the integral average of $f$ over $Q$ by 
\begin{equation}
\vint_{Q}f(w)\, dw:=\frac{1}{|Q|}\int_{Q}f(w)\, dw\, .\nonumber
\end{equation} 

Let us give the definitions of Sobolev space $W^{1,p}(U)$ and Haj\l asz-Sobolev space $M^{1,p}(U)$.
\begin{defn}\label{Sobolev}
We define the first order Sobolev space $W^{1,p}(U)$, $1\leq p\leq \fz$, as the set 
\begin{equation}
\left\{\, u\in L^p(U) \setcolon \nabla u\in L^p(U\setcolon\rn)\, \right\}\, .\nonumber
\end{equation}
Here $\nabla u=\left(\frac{\partial u}{\partial x_1}\, ,\, \dots \, ,\,\frac{\partial u}{\partial x_n}\right)$ is the weak (or distributional) gradient of a locally integrable function $u$. 
\end{defn}
We equip $W^{1,p}(U)$ with the norm:
\begin{equation}
\|u\|_{W^{1, p}(U)}=\|u\|_{L^p(U)}+\|\nabla u\|_{L^p(U)}\nonumber
\end{equation}
for $1\leq p\leq\fz$, where $\|\cdot\|_{L^p(U)}$ denotes the usual $L^p$-norm for $p \in [1,\infty]$. The following lemma from \cite[Page 13]{mazya} tells us that on a star-shaped domain, Sobolev functions can be approximated by global smooth functuons.
\begin{lemma}\label{le:approx}
Let $\boz\subset\rn$ be a star-shaped domain. Then $C^\fz(\overline\boz)\cap W^{1, p}(\boz)$ is dense in $W^{1, p}(\boz)$ for $1\leq p\leq\fz$.
\end{lemma}

For $u\in\lp$, we denote by $\mathcal{D}_p(u)$ the class of functions $0\leq g\in\lp$ for which there exists $E\subset U$ with $|E|=0$, so that 
\begin{equation}
|u(z_1)-u(z_2)|\leq |z_1-z_2|\left(g(z_1)+g(z_2)\right), \ \ {\rm for} \ \ z_1, z_2\in U\setminus E\, .\nonumber
\end{equation}
\begin{defn}\label{hajlasz}
We define the Haj\l asz-Sobolev space $M^{1,p}(U)$, $1\leq p\leq \fz$, as the set 
\begin{equation}
\left\{u\in L^p(U), \mathcal{D}_p(u)\neq\emptyset\right\}\, .\nonumber
\end{equation}
\end{defn}
We equip $M^{1,p}(U)$ with the norm:
\begin{equation}
\|u\|_{M^{1, p}(U)}=\|u\|_{L^p(U)}+\inf_{g\in\mathcal{D}_p(u)}\|g\|_{L^p(U)}\, .\nonumber
\end{equation}
for $1\leq p\leq\fz$. For $1<p\leq\fz$, we write $W^{1, p}(U)=_C M^{1, p}(U)$ if we have $W^{1, p}(U)=M^{1, p}(U)$ with 
\[\frac{1}{C}\|\nabla u\|_{L^p(U)}\leq\inf_{g\in\mathcal D_p(u)}\|g\|_{L^p(U)}\leq C\|\nabla u\|_{L^p(U)}\]
for a positive constant $C>1$ independent of $u\in W^{1, p}(U)$. The following lemma tells us that the equivalence of the Sobolev space and the Haj\l{}asz-Sobolev space on bounded star-shaped domain is invariant under linear stretching.
\begin{lemma}\label{le:uniform}
Let $\boz\subset\rn$ be a bounded star-shaped domain with $W^{1, p}(\boz)=_C M^{1, p}(\boz)$ for some $1<p\leq\fz$. Then for every $0<\lambda<\fz$, we have $W^{1, p}(\lambda\boz)=_{C} M^{1, p}(\lambda\boz)$ with a same constant $C$.
\end{lemma}
\begin{proof}
Without loss of generality, we may assume $0\in\boz$ is a star center. Fix $0<\lambda<\fz$. Let $u\in W^{1,p}(\lambda\boz)$ be arbitrary. We define a function $u_\lambda$ on $\boz$ by setting 
\[u_\lambda(z):=u(\lambda z)\]
for every $z\in\boz$. Then, by the change of variables formula, we have $u_\lambda\in W^{1, p}(\boz)$ with
\begin{equation}\label{Equa1}
 \|\nabla u_\lambda\|_{L^p(\boz)}=\lambda^{1-\frac{n}{p}}\|\nabla u\|_{L^p(\lambda\boz)}.
\end{equation}
Since $W^{1, p}(\boz)=_C M^{1, p}(\boz)$, there exists a function $g_{u_\lambda}\in\mathcal D_p(u_\lambda)$ with 
\begin{equation}\label{eq:compare}
\frac{1}{C}\|\nabla u_\lambda\|_{L^p(\boz)}\leq\|g_{u_\lambda}\|_{L^p(\boz)}\leq C\|\nabla u_\lambda\|_{L^p(\boz)}.
\end{equation}
 Then we define a function $g_u$ on $\lambda\boz$ by setting 
\[g_u(z):=\frac{1}{\lambda}g_{u_\lambda}\lf(\frac{z}{\lambda}\r)\]
for every $z\in\lambda\boz$. Then for almost every $z_1, z_2\in\lambda\boz$, we have 
\begin{eqnarray}
|u(z_1)-u(z_2)|&=&\lf|u_\lambda\lf(\frac{z_1}{\lambda}\r)-u_{\lambda}\lf(\frac{z_2}{\lambda}\r)\r|\nonumber\\
                   &\leq&|z_1-z_2|\lf(\frac{1}{\lambda}g_{u_\lambda}\lf(\frac{z_1}{\lambda}\r)+\frac{1}{\lambda}g_{u_\lambda}\lf(\frac{z_2}{\lambda}\r)\r)\nonumber\\
                   &\leq&|z_1-z_2|(g_u(z_1)+g_u(z_2)).\nonumber
\end{eqnarray}
The change of variables formula implies 
\begin{equation}\label{Equa2}
\|g_u\|_{L^p(\lambda\boz)}=\lambda^{\frac{n}{p}-1}\|g_{u_\lambda}\|_{L^p(\boz)}.
\end{equation}
Hence $g_u\in\mathcal D_p(u)$. By combining inequalities (\ref{Equa1}), (\ref{eq:compare}) and (\ref{Equa2}), we obtain the desired inequality
\[\frac{1}{C}\|\nabla u\|_{L^p(\lambda\boz)}\leq\|g_u\|_{L^p(\lambda\boz)}\leq C\|\nabla u\|_{L^p(\lambda\boz)}.\]
\end{proof}

\section{Maximal functions}
We will define a maximal function $M^{\tau}[f]$. That will vary only the first component $t$. 
For every $x\in\psi(1)\boz\subset\rr^{n-1}$ set 
\begin{equation}
S_x:=\{t\in\rr\setcolon (t,x)\in\Omega_\psi\} .\nonumber
\end{equation}
%\endblue
%\colgreen
Let
%\endgreen
 $f\colon\Omega_{\psi}\to\rr$ be
measurable and let $(t, x)\in \Omega_\psi$. We define the one-dimensional maximal function in the direction of the first variable by setting
%\colblue
\begin{equation}\label{t-max fun}
M^{\tau}[f](t, x):=\sup_{[a,b]\ni t}\vint_{[a, b]\cap S_x}|f(s,x)|\, ds\, .
\end{equation}
%\endblue
%\colred
%The measure $ds$ is the usual length measure on the one-dimensional slice $\Omega_\psi \cap (\R \times \{x\})$, and 
%\endred
The supremum is taken over all intervals $[a, b]$ containing $t$. 

The next lemmas tell us that $M^\tau$ enjoys the usual $L^p$-boundedness property. See \cite[Lemma 3.1]{IMRN} for a proof.
\begin{lemma}\label{lem:M}
Let $1<p<\fz$. Then for every $f\in L^p(\boz_\psi)$, $M^\tau[f]$ is measurable and we have 
\begin{equation}\label{equa1}
\int_{\Omega_\psi }\lf|M^\tau[f](z)\r|^p\, dz\leq C\int_{\Omega_\psi }\lf|f(z)\r|^p\, dz\, ,
\end{equation}
where the constant $C$ is independent of $f$. 
\end{lemma}
%\begin{proof}
%Since the maximal function comes out the same if we consider only segments with rational endpoints,
%it preserves measurability.
%\colred
%For every $x\in B^{n-1}(0,\psi(1))$ set 
%\begin{equation}
%S_x:=\{t\in\rr\setcolon (t,x)\in\Omega\} .\nonumber
%\end{equation}
%\endred
%Fubini's theorem implies that $f(\cdot,x)\in L^p(S_x)$ for almost every $x\in\psi(1)\boz$. By the $L^p$-boundedness of the classical Hardy-Littlewood maximal function on the interval $S_x$, for such $x$ 
%we have 
%\begin{equation}\label{equa:Lpbound}
%\int_{S_x}\lf|M^{\tau}[f](t,x)\r|^p\, dt\leq C\int_{S_x}\lf|f(t,x)\r|^p\, dt,
%\end{equation}
%where the constant $C$ is independent of $f$ and $x$. By combining the inequality (\ref{equa:Lpbound}) and Fubini's theorem, we obtain 
%\begin{eqnarray}
%\int_{\boz_\psi}\lf|M^\tau[f](t, x)\r|^p\, dx\, dt
%&=&\int_{\psi(1)\boz}\int_{S_x}\lf|M^\tau[f](t, x)\r|^p\, dt\, dx\nonumber\\
%&\leq&C\int_{\psi(1)\boz}\int_{S_x}\lf|f(t, x)\r|^p\, dt\, dx\nonumber\\
%&=&C\int_{\boz_\psi}\lf|f(t, x)\r|^p\, dx\, dt\, .\nonumber
%\end{eqnarray}
%\end{proof}

\section{Proof of the Main theorem}
Let $\boz\subset\rr^{n-1}$ be a star-shaped bounded domain with $W^{1, p}(\boz)=M^{1, p}(\boz)$ for some $1<p<\fz$. {Then $\boz_\psi\subset\rn$ is also a star-shaped domain. By Lemma \ref{le:approx}, $C^\fz(\overline{\boz_\psi})\cap W^{1, p}(\boz_\psi)$ is dense in $W^{1, p}(\boz)$.} Let $u\in C^\fz(\overline{\boz_\psi})\cap W^{1,p}(\boz_\psi)$ be arbitrary. Fix $0<t<2$, define the restriction of $u$ to $\{t\}\times\psi(t)\boz$ by setting 
\begin{equation}\label{howtildeu}
u_t(x)=u(t, x)\ {\rm on\ every}\ x\in\psi(1)\boz.
\end{equation}
Then $u_t\in W^{1, p}(\psi(t)\boz)$ for every $0<t<2$. We will construct a nonnegative function $g_u\in L^p(\boz_\psi)$ such that for every $t\in(0, 2)$ and every $x, y\in\psi(t)\boz$ we have 
\[|u_t(x)-u_t(y)|\leq |x-y|(g_u(t, x)+g_u(t, y))\]
and 
\[\int_{\boz_\psi}g_u^p(z)dz\leq C\int_{\boz_\psi}|\nabla u(z)|^pdz\]
with a constant $C$ independent of $u$. {If 
\[\int_{\boz_\psi}|\nabla u(z)|^pdz=0,\]
then $u\equiv c$ on $\boz_\psi$ for some constant $c\in\rr$. In this case, we simply define $g_u\equiv 0$ on $\boz_\psi$. Then we have 
\[|u(z_1)-u(z_2)|\leq |z_1-z_2|(g_u(z_1)+g_u(z_2))\]
for every $z_1, z_2\in\boz_\psi$ and $$\|g_u\|_{L^p(\boz_\psi)}=\|\nabla u\|_{L^p(\boz_\psi)}.$$ Let us consider the case that
\[T_u:=\int_{\boz_\psi}|\nabla u(z)|^pdz>0.\]
Denote the gradient with respect to the $x$-variable by $\nabla^{\chi}$. By Lemma \ref{le:uniform}, for every $t\in (0, 2)$, there exists a nonnegative function $g_t\in L^p(\psi(t)\boz)$ with%{\red 
%\begin{equation}\label{eq:lbound}
%|\nabla^\chi u_t(x)|\leq g_t(x)
%\end{equation}
%for almost every $x\in\psi(t)\boz$,}
\begin{equation}\label{eq:est0}
|u(x_1)-u(x_2)|\leq|x_1-x_2|(g_{t}(x_1)+g_{t}(x_2))\, 
\end{equation}
for almost every\ $x_1$, $x_2\in\psi(t)\boz$ and
\begin{equation}\label{eq:control1}
\|g_{t}\|_{L^p(\psi(t)\boz)}\leq C\|\nabla^{\chi}u_t\|_{L^p(\psi(t)\boz)}
\end{equation}
for a constant $C$ independent of $t$. {Simply reset $g_t$ to be $\fz$ on a measure zero set, we can assume inequality (\ref{eq:est0}) hold for every $x_1, x_2\in\psi(t)\boz$. We define 
\[\hat g_t(x):=2g_t(x)+(T_u)^{\frac{1}{p}}\]
for every $x\in\psi(t)\boz$.  Since $u\in C^\fz(\overline{\boz_\psi})\cap W^{1, p}(\boz_\psi)$, there exists a small enough $0<\delta<1$ such that for every $x, y\in\psi(t)\boz$ with $0<|x-y|<\delta$, we have 
\begin{equation}\label{EQ:control1}
|u_t(x)-u_t(y)|< |x-y|(g_t(x)+(T_u)^{\frac{1}{p}})\leq|x- y|(\hat g_t(x)+\hat g_t(y)).
\end{equation}
Since $\psi(s)\boz\subset\psi(t)\boz$ for every $0<s<t$, there exists a small enough $0<\epsilon_t^1<t$ such that for every $s\in(t-\epsilon_t^1, t]$ and every $x, y\in\psi(s)\boz$ with $|x-y|<\delta$, we have 
\begin{equation}\label{EQ:control2}
|u_s(x)-u_s(y)|<|x-y|(\hat g_t(x)+\hat g_t(y))
\end{equation} 
Due to $u\in C^\fz(\overline{\boz_\psi})\cap W^{1, p}(\boz_\psi)$ again,  there exists a small enough $0<\epsilon^2_t<t$ such that for every $s\in(t-\epsilon_t^2, t]$ and every $x, y\in\psi(s)\boz$ with $|x-y|\geq\delta$, we have 
\begin{equation}\label{EQ:control3}
|u_s(x)-u_s(y)|\leq|x-y|(\hat g_t(x)+\hat g_t(y)).
\end{equation}
Hence, we can find a sufficiently small $0<\epsilon^t\leq\min\{\epsilon_t^1, \epsilon_t^2\}$ such that for every $s\in(t-\epsilon_t, t]$ and every $x, y\in \psi(s)\boz$, we have
\begin{equation}\label{EQ:control4}
|u_s(x)-u_s(y)|\leq|x-y|(\hat g_t(x)+\hat g_t(y)),
\end{equation}
and for every $s\in(t-\epsilon_t, t]$, we have 
\begin{equation}\label{EQ:control5}
\|g_t\|_{L^p(\psi(s)\boz)}\leq C\|\nabla^\chi u_s\|_{L^p(\psi(s)\boz)}
\end{equation}
with a constant $C$ independent of $s$ and $t$. By the covering theorem on the line segment $(0, 2]$, there exists an at most countable class $\lf\{(t_i-\epsilon_{t_i}, t_i]\r\}_{i\in I\subset\mathbb N}$ such that 
\[(0, 2]\subset\bigcup_{i\in I}(t_i-\epsilon_{t_i}, t_i]\]
and 
\[\sum_{i\in I}\chi_{(t_i-\epsilon_{t_i}, t_i]}(t)\leq 2\]
for every $t\in(0, 2]$. Simply extend $\hat g_{t}$ to $\rr^{n-1}$ by setting it to be $0$ outside $\psi(t)\boz$ and define a function $g_u$ on $\boz_\psi$ by setting 
\begin{equation}\label{eq:Upper}
g_u(t, x):=\sum_{i\in I} \hat g_{t_i}(x)\chi_{(t_i-\epsilon_{t_i}, t_i]}(t)
\end{equation}
for every $z=(t, x)\in\boz_\psi$. By (\ref{EQ:control4}) and (\ref{eq:Upper}), for every $t\in(0, 2]$ and every $x, y\in\psi(t)\boz$, we have 
\begin{equation}\label{EQ:control5}
|u_t(x)-u_t(y)|\leq|x-y|(g_u(t, x)-g_u(t, y)).
\end{equation}
By the argument above, we obtain $g_u\in L^p(\boz_\psi)$ with 
\begin{multline}
\int_{\boz_\psi}g_u^p(z)dz\leq C\sum_{i\in I}\int_{t_i-\epsilon_{t_i}}^{t_i}\int_{\psi(t)\boz}\hat g_{t_i}^p(x)dxdt\nonumber\\
\leq C\sum_{i\in I}\int_{t_i-\epsilon_{t_i}}^{t_i}\int_{\psi(t)\boz}(g^p_{t_{i}}(x)+T_u)dxdt\nonumber\\
\leq C\int_0^2\int_{\psi(t)\boz}|\nabla^\chi u_t(x)|^pdxdt+T_u
\leq C\int_{\boz_\psi}|\nabla u(z)|^pdz.\nonumber
\end{multline}
 }

First, we introduce some results which will be used in the proof of that $W^{1, p}(\boz)=M^{1, p}(\boz)$ implies $W^{1, p}(\boz_\psi)=M^{1, p}(\boz_\psi)$. By \cite{Pitor}, there is a bounded inclusion $\iota \co M^{1,p}(\boz_\psi) \hookrightarrow W^{1,p}(\boz_\psi)$.  To show that $\iota$ is an isomorphism, it suffices to show that its inverse $\iota^{-1}$ is both densely defined and bounded on $W^{1,p}(\boz_\psi)$. Since $C^\fz(\overline{\boz_\psi})\cap W^{1,p}(\boz_\psi)$ is dense in $W^{1,p}(\boz_\psi)$, it suffices to show that $C^\fz(\overline{\boz_\psi})\cap W^{1,p}(\boz_\psi) \subset M^{1, p}(\boz_\psi)$ and that for each $u \in C^\fz(\overline{\boz_\psi})\cap W^{1,p}(\boz_\psi)$ we have $$||u||_{M^{1,p}(\Omega_\psi)} \leq C||u||_{W^{1,p}(\boz_\psi)},$$ for a positive constant independent of $u$. Next, we prove the main estimate, which is similar to Lemma $4.1$ in \cite{IMRN}. To gain a proof for the following lemma, we simply replace $M^\chi[|\nabla u|]$ there by $g_u$ and repeat the argument.
\begin{lemma}\label{l:est1}
Let $\boz\subset\rr^{n-1}$ be a bounded star-shaped domain with $W^{1, p}(\boz)=M^{1, p}(\boz)$ for $1<p<\fz$ and $\psi:(0, 1]\to(0, \fz)$ be a left continuous and increasing function. Define an outward cuspidal domain $\boz_\psi$ as in (\ref{cusp}). Let $z_1=(t_1, x_1), z_2:=(t_2, x_2)\in\boz_\psi$ be two points with $t_1<t_2$.  Suppose that $u \in W^{1,p}(\Omega_{\psi})\cap C^1(\boz_\psi)$.
Then we have
\begin{eqnarray}\label{eq:est1}
|u(z_1)- u(z_2)| &\leq& C|z_1-z_2|\big(M^{\tau}[|
%\colgreen
\nabla u
%\endgreen
|](z_1) \ +\  M^{\tau}[g_{u}](z_1) \ +\ \nonumber \\
&&\ \ \ \ \ \ \ \ \  \ \ \ \ \,  M^{\tau}[|
%\colgreen
\nabla u
%\endgreen
|](z_2) \ +\  M^{\tau}[g_{u}](z_2)\big)\,.
\end{eqnarray}
%\begin{equation}\label{eq:est1}
%|u(z_1)- \tilde u_{Q_{z_1,z_2}}| \leq C|z_1-z_2|(M^{1}[\nabla u](z_1) + M^{1}[M^\chi[\nabla u]](z_1))
%\end{equation}
%and
%\begin{equation}\label{eq:est2}
%|u(z_2)-\tilde u_{Q_{z_1, z_2}}|\leq C|z_1-z_2|(M^{1}[\nabla u](z_2)+M^{1}[M^\chi[\nabla u]](z_2)).
%\end{equation}
\end{lemma}

Let us prove the main result in this note.
\begin{proof}[Proof of Theorem \ref{thm:main}]  By \cite[Theorem 7]{hkt}, if $U$ is a bounded domain, $W^{1, \fz}(U)=M^{1, \fz}(U)$ if and only if $U$ is quasiconvex. Recall that a domain $U$ is quasiconvex if there exists a constant $C \geq 1$ such that, for every pair of points $x,y \in U$, there is a rectifiable curve $\gamma \subset U$ joining $x$ to $y$ so that ${\rm len}(\gamma) \leq C|x-y|$ for a constant $C$ independent of $x, y$. For every left continuous and increasing function $\psi:(0, 1]\to(0, \fz)$, $\boz_\psi$ is quasiconvex if and only if $\boz$ is quasiconvex. Hence, we have $W^{1, \fz}(\boz_\psi)=M^{1, \fz}(\boz_\psi)$ if and only if $W^{1, \fz}(\boz)=M^{1, \fz}(\boz)$. 

Fix $1<p<\fz$. First, we show $M^{1, p}(\boz)=W^{1, p}(\boz)$ implies $M^{1, p}(\boz_\psi)=W^{1, p}(\boz_\psi)$ By \cite{Pitor}, we know that $M^{1,p}(\boz_\psi)$ can be boundedly embedded into $W^{1,p}(\boz_\psi)$. To show $W^{1, p}(\boz_\psi)=M^{1, p}(\boz_\psi)$ it suffices to show that the dense subspace $C^\fz(\overline{\boz_\psi})\cap W^{1,p}(\boz_\psi)$ of $W^{1,p}(\boz_\psi)$ is contained in $M^{1,p}(\boz_\psi)$ with $M^{1, p}$-norm is controlled by $W^{1, p}$-norm form above uniformly. Let $u\in C^{\fz}(\overline{\boz_\psi})\cap W^{1,p}(\boz_\psi)$ be arbitrary.
Set
\begin{equation}\label{equa:ud}
\hat g(t, x)=M^\tau[|\nabla  u|](t, x)+g_{u}(t, x)
+M^\tau[g_{u}](t, x)\, .
\end{equation}
Here $g_u$ is defined as in (\ref{eq:Upper}).
%\colred
%Since  $M^\chi[|\nabla \tilde u|](z)$ is lower-semi continuous, we have
%\begin{equation}
%M^\chi[|\nabla \tilde u|](z)\leq M^{\tau}[M^\chi[|\nabla^{\chi} \tilde u|]](z)\, \nonumber
%\end{equation}
%for every $z\in\boz_\psi$. 
%\endred

By \eqref{eq:est0} and Lemma \ref{l:est1},  for every $z_1, z_2\in\boz_\psi$, we get the estimate
\begin{equation}
|u(z_1)-u(z_2)|\leq C|z_1-z_2|(\hat g(z_1)+\hat g(z_2))\, .\nonumber
\end{equation}
Define $g:=C\hat g\in\mathcal D_p(u)$ for a suitable constant $C>1$. The triangle inequality gives
\begin{equation}
\int_{\boz_\psi}|g(z)|^p\, dz\leq C\lf(\int_{\boz_\psi}M^\tau[|\nabla  u|](z)^p\, dz+\int_{\boz_\psi}g_{u}(z)^p\,dz+\int_{\boz_\psi}M^\tau[g_{u}](z)^p\, dz\r)\, .\nonumber
\end{equation}
The inequality (\ref{eq:control1}) leads to the estimate
\begin{equation}
\int_{\boz_\psi}g_{u}(z)^p\, dz \leq C\int_{\boz_\psi}|\nabla^{\chi}u(z)|^pdz.
\end{equation}
Lemma \ref{lem:M} leads to the estimates
 \begin{equation}
 \int_{\boz_\psi}|M^\tau[|\nabla  u|](z)|^p\, dz\leq C \int_{\boz_\psi}|\nabla  u(z)|^p\, dz\nonumber
 \end{equation}
 and 
 \begin{eqnarray}
 \int_{\boz_\psi}|M^\tau[g_{u}](z)|^p\, dz&\leq& C\int_{\boz_\psi}g_{u}(z)^p\, dz\nonumber\\
&\leq& C\int_{\boz_\psi}|\nabla u(z)|^p\, dz\, ,\nonumber
 \end{eqnarray}
 which imply that $g\in\mathcal{D}_p(u)$ and that $$\|u\|_{M^{1,p}(\boz_\psi)} \leq C \|u\|_{W^{1,p}(\boz_\psi)}.$$ That is, $C^\fz(\overline{\boz_\psi})\cap W^{1,p}(\boz_\psi)$ can be boundedly embedded into $M^{1, p}(\boz_\psi)$. Hence, we proved that $W^{1, p}(\boz)=M^{1, p}(\boz)$ implies $W^{1, p}(\boz_\psi)=M^{1, p}(\boz_\psi)$.
 
Next, we prove $W^{1, p}(\boz_\psi)=M^{1, p}(\boz_\psi)$ implies $W^{1, p}(\boz)=M^{1, p}(\boz)$. {Since $\boz$ is star-shaped, by a similar argument as above}, it suffices to show the dense subspace $C^\fz(\overline{\boz})\cap W^{1,p}(\boz)$ can be boundedly embedded into $M^{1, p}(\boz)$. Let $u\in C^\fz(\overline\boz)\cap W^{1, p}(\boz)$ be arbitrary. If $u\equiv c$ for some constant $c\in\rr$, then  $u\in M^{1,p}(\boz)$ with $$\|u\|_{W^{1, p}(\boz)}=\|u\|_{M^{1, p}(\boz)}.$$ Hence, we assume $u$ is not a constant function. Hence, we have 
$$\|\nabla u\|_{L^p(\boz)}>0.$$ 
As before, we assume $0\in\boz$ is a star center. We define a function $\tilde u$ on $\psi(1)\boz$ by setting 
\[\tilde u(x)=u\lf(\frac{x}{\psi(1)}\r)\ {\rm for\ every}\ x\in\psi(1)\boz.\]
The change of variables formula implies
\begin{equation}\label{eq:inequa1}
\|\tilde u\|_{L^p(\psi(1)\boz)}=\psi(1)^{\frac{-n}{p}}\|u\|_{L^p(\boz)}\ {\rm and}\ \|\nabla^\chi\tilde u\|_{L^p(\psi(1)\boz)}=\psi(1)^{1-\frac{n}{p}}\|\nabla^\chi u\|_{L^p(\boz)}.
\end{equation}
Hence, $\tilde u\in C^1(\psi(1)\boz)\cap W^{1, p}(\psi(1)\boz)$. Simply by the geometry, we can write
\[\boz_\psi:=\bigcup_{x\in\psi(1)\boz}S_x.\]
We define a function $\hat u$ on $\boz_\psi$ by setting 
\[\hat u(t, x):=\tilde u(x)\ {\rm for\ every}\ (t, x)\in\boz_\psi.\]
Since $\psi(t)\boz\subset\psi(1)\boz$ for every $t\in(0, 2)$, we have $\hat u\in C^1(\boz_\psi)$ with
\begin{equation}
\|\hat u\|_{L^p(\boz_\psi)}\leq2\|\tilde u\|_{L^p(\psi(1)\boz)}\ {\rm and}\ \|\nabla\hat u\|_{L^p({\boz_\psi})}\leq 2\|\nabla^\chi\tilde u\|_{L^p(\psi(1)\boz)}.\nonumber
\end{equation}
Hence, $\hat u\in C^1(\boz_\psi)\cap W^{1, p}(\boz_\psi)$. Since $W^{1, p}(\boz_\psi)=M^{1, p}(\boz_\psi)$, there exists $g\in\mathcal D_p(\hat u)$ with 
\[\|g\|_{L^p(\boz_\psi)}\leq C\|\nabla\hat u\|_{L^p(\boz_\psi)}\]
and
\[|\hat u(z_1)-\hat u(z_2)|\leq|z_1-z_2|(g(z_1)+g(z_2))\]
for almost every $z_1, z_2\in\boz_\psi$. We assume last inequality holds at every $z_1, z_2\in\boz_\psi$ by simply setting $g=\fz$ on a measure-zero set. Set $g_t$ to be the restriction of $g$ to $\{t\}\times\psi(t)\boz$ and define
\[A:=\inf_{t\in(1, 2)}\|g_t\|_{L^p(\psi(t)\boz)}.\]
Then, we have 
\[0<A\leq C\|\nabla\hat u\|_{L^p(\boz_\psi)}\leq C\|\nabla^\chi\tilde u\|_{L^p(\psi(1)\boz)}.\]
There exists $\hat t\in(1, 2)$ with $$A\leq\|g_{\hat t}\|_{L^p(\psi(\hat t)\boz)}\leq2A.$$ 
Then for every $x_1,x_2\in\psi(1)\boz$, we have 
\[|\tilde u(x_1)-\tilde u(x_2)|=|\hat u(\hat t, x_1)-\hat u(\hat t, x_2)|\leq |x_1-x_2|(g_{\hat t}(x_1)+g_{\hat t}(x_2)).\]
Hence, we have $g_{\hat t}\in\mathcal D_p(\tilde u)$ with 
\[\|g_{\hat t}\|_{L^p(\psi(1)\boz)}\leq C\|\nabla^\chi\tilde u\|_{L^p(\psi(1)\boz)}.\]
Define a function $g$ on $\boz$ by setting 
\[g(x):=\frac{1}{\psi(1)}g_{\hat t}(\psi(1)x)\ {\rm for\ every}\ x\in\boz.\]
Then, we have 
\[|u(x_1)-u(x_2)|\leq|x_1-x_2|(g(x_1+g(x_2)))\]
for every $x_1,x_2\in\boz$, and
\begin{equation}\label{eq:inequa2}
\|g\|_{L^p(\boz)}=\psi(1)^{\frac{n}{p}-1}\|g_{\hat t}\|_{L^p(\psi(1)\boz)}.
\end{equation}
Hence, we obtain $g\in\mathcal D_p(u)$ with 
\[\|g\|_{L^p(\boz)}\leq C\|\nabla^\chi u\|_{L^p(\boz)}\]
for a constant $C$ independent of $u$. Hence, we have $C^\fz(\overline{\boz})\cap W^{1,p}(\boz)\subset M^{1, p}(\boz)$ with
\[\|u\|_{M^{1, p}(\boz)}\leq C\|u\|_{W^{1, p}(\boz)}\]
as desired. 
\end{proof}

\end{document}